\documentstyle[12pt]{article}
\newtheorem{theorem}{Theorem}
\newtheorem{lemma}[theorem]{Lemma}
\newtheorem{proposition}[theorem]{Proposition}
\newtheorem{examp}{Example}
\newtheorem{corollary}[theorem]{Corollary}
\newtheorem{remar}[theorem]{Remark}
\newtheorem{ex}{Exercise}
\newtheorem{prob}{Open Problem}
\newenvironment{proof}{Proof:\ \ \ }{\QED}

\newenvironment{example}{\begin{examp}\rm}{\diams\end{examp}}
\newenvironment{remark}{\begin{remar}\rm}{\end{remar}}

\newcommand{\diams}{\unskip\nobreak\hfil\penalty50%
\hskip1em\hbox{}\nobreak\hfil%
$\diamondsuit$\parfillskip=0pt\finalhyphendemerits=0}
\newcommand{\QED}{{\unskip\nobreak\hfil\penalty50%
\hskip1em\hbox{}\nobreak\hfil $\Box$%
\parfillskip=0pt \finalhyphendemerits=0 \par\medskip\noindent}}
\newcommand{\bfind}[1]{\index{#1}{\bf #1}}

\newcommand{\n}{\par\noindent}

\newcommand{\sn}{\par\smallskip\noindent}

\newcommand{\bn}{\par\bigskip\noindent}
\newcommand{\pars}{\par\smallskip}
\newcommand{\parm}{\par\medskip}
\newcommand{\parb}{\par\bigskip}
\newcommand{\isom}{\simeq}

\newcommand{\ovl}[1]{\overline{#1}}
\newcommand{\chara}{\mbox{\rm char}\,}
\newcommand{\trdeg}{\mbox{\rm trdeg}\,}

\newcommand{\Quot}{\mbox{\rm Quot}\,}
\newcommand{\PSF}{\mbox{\rm PSF}\,}
\newcommand{\adresse}{\par\bigskip \small\rm
   Mathematical Sciences Group, 
   University of Saskatchewan, \par
   106 Wiggins Road, 
   Saskatoon, Saskatchewan, Canada S7N 5E6 \par
   email: fvk@math.usask.ca }
\font\tenlv=msbm10 scaled 1200
\font\sevenlv=msbm7 scaled 1200
\font\fivelv=msbm5 scaled 1200
\def\lv #1{{\mathchoice{{\hbox{\tenlv #1}}}{{\hbox{\tenlv #1}}}
{{\hbox{\sevenlv #1}}}{{\hbox{\fivelv #1}}}}}

\newcommand{\N}{\lv N}
\newcommand{\Q}{\lv Q}

\newcommand{\Z}{\lv Z}
\newcommand{\F}{\lv F}

\begin{document}
\title{Dense subfields of henselian fields, and integer
parts\footnote{This paper was written while I was a guest of the Equipe
G\'eom\'etrie et Dynamique, Institut Math\'ematiques de Jussieu, Paris,
and of the Equipe Alg\`ebre--G\'eom\'etrie at the University of
Versailles. I gratefully acknowledge their hospitality and support.
I was also partially supported by a Canadian NSERC grant and by a
sabbatical grant from the University of Saskatchewan. Furthermore I am
endebted to the organizers of the conference in Teheran and the members
of the IPM and all our friends in Iran for their hospitality and
support. I also thank the two referees as well as A.\ Fornasiero for
their careful reading of the paper and their useful suggestions. This
paper is dedicated to Salma Kuhlmann who got me interested in the
subject and provided the personal contacts that inspired and supported
my work.}}
\author{Franz-Viktor Kuhlmann}
\date{26.\ 6.\ 2005}
\maketitle
\begin{abstract}\noindent
{\footnotesize\rm
We show that every henselian valued field $L$ of residue characteristic
$0$ admits a proper subfield $K$ which is dense in $L$. We present
conditions under which this can be taken such that $L|K$ is
transcendental and $K$ is henselian. These results are of interest for
the investigation of integer parts of ordered fields. We present
examples of real closed fields which are larger than the quotient fields
of all their integer parts. Finally, we give rather simple examples of
ordered fields that do not admit any integer part and of valued fields
that do not admit any subring which is an additive complement of the
valuation ring.}
\end{abstract}
%
%
%
%
\section{Introduction}
At the ``Logic, Algebra and Arithmetic'' Conference, Teheran 2003,
Mojtaba Moniri asked the following question: {\it Does every
non-archimedean ordered real closed field $L$ admit a proper dense
subfield $K$?} This question is interesting since if such a subfield $K$
admits an integer part $I$ then $I$ is also an integer part for $L$, but
the quotient field of $I$ lies in $K$ and is thus smaller than $L$. An
\bfind{integer part} of an ordered field $K$ is a discretely ordered
subring $I$ with $1$ such that for all $a\in K$ there is $r\in I$ such
that $r\leq a<r+1$. It follows that the element $r$ is uniquely
determined, and in particular that $1$ is the least positive element in
$I$.

Since the natural valuation of a non-archimedean ordered real closed
field $L$ is non-trivial, henselian and has a (real closed) residue
field $Lv$ of characteristic $0$, the following theorem answers the
above question to the affirmative:

\begin{theorem}                             \label{MT1}
Every henselian non-trivially valued field $(L,v)$ with a residue field
of characteristic $0$ admits a proper subfield $K$ which is dense in
$(L,v)$. This subfield $K$ can be chosen such that $L|K$ is algebraic.
\end{theorem}

Here, density refers to the topology induced by the valuation; that is,
$K$ is dense in $(L,v)$ if for every $a\in L$ and all values $\alpha$ in
the value group $vL$ of $(L,v)$ there is $b\in K$ such that $v(a-b)\geq
\alpha$. In the case of non-archimedean ordered fields with natural (or
non-trivial order compatible) valuation, density in this sense is
equivalent to density with respect to the ordering.

In the case where the value group $vL$ has a maximal proper convex
subgroup, the proof is quite easy, but does in general not render any
subfield $K$ such that $L|K$ is transcendental. In the case of $vL$
having no maximal proper convex subgroup, the proof is much more
involved, but leaves us the choice between $L|K$ algebraic or
transcendental:

\begin{theorem}                             \label{MT2}
In addition to the assumptions of Theorem~\ref{MT1}, suppose that $vL$
does not have a maximal proper convex subgroup. Then for each integer
$n\geq 1$ there is a henselian (as well as a non-henselian) subfield $K$
dense in $L$ such that $\trdeg L|K= n$. It can also be chosen such that
$\trdeg L|K$ is infinite.
\end{theorem}

To see that such valued fields $(L,v)$ exist, take $x_i\,$, $i\in\N$, to
be a set of algebraically independent elements over an arbitrary field
$k$ and define a valuation $v$ on $k(x_i\mid i\in\N)$ by setting
$0<vx_1\ll vx_2\ll\ldots\ll vx_i\ll\ldots$; then pass to the
henselization of $(k(x_i\mid i\in\N),v)$. For a more general approach,
see Lemma~\ref{exvgcof}.

\begin{remark}
A.\ Fornasiero [F] has shown that every henselian valued field with a
residue field of characteristic $0$ admits a truncation closed embedding
in a power series field with coefficients in the residue field and
exponents in the value group (in general, the power series field has to
be endowed with a non-trivial factor system). ``Truncation closed''
means that every truncation of a power series in the image lies again in
the image.

It follows that all of the henselian dense subfields admit such
truncation closed embeddings. But also the dense non-henselian subfields
can be chosen such that they admit truncation closed embeddings. We will
sketch the proof in Section~\ref{sectds} (Remarks~\ref{remtce}
and~\ref{r2}).
\end{remark}

Our construction developed for the proof of Theorem~\ref{MT2} also gives
rise to a counterexample to a quite common erroneous application of
Hensel's Lemma. A valuation $w$ is called a \bfind{coarsening} of $v$ if
its associated valuation ring contains that of $v$. In this case, $v$
induces a valuation $\ovl{w}$ on the residue field $Kw$ whose valuation
ring is simply the image of the valuation ring of $v$ under the residue
map associated with $w$. The counterexample proves:

\begin{proposition}                           \label{cor1}
There are valued fields $(K,v)$ such that $vK$ has no maximal proper
convex subgroup, the residue field $(Kw,\ovl{w})$ is henselian for every
non-trivial coarsening $w\ne v$ of $v$, but $(K,v)$ itself is not
henselian.
\end{proposition}

The proofs of Theorems~\ref{MT1} and~\ref{MT2} and of
Proposition~\ref{cor1} are given in Section~\ref{sectds}. There, we will
also give a more explicit version of Theorem~\ref{MT2}.

\parm
In general, the quotient fields of integer parts of an ordered field are
smaller than the field. The following theorem will show that there are
real closed fields for which the quotient field of {\it every} integer
part is a proper subfield. If $k$ is any field, then
\[\PSF(k)\;:=\;\bigcup_{n\in\N} k((t^{\frac{1}{n}}))\]
is called the \bfind{Puiseux series field over $k$}; it is a subfield of
the power series field $k((t^{\Q}))$ with coefficients in $k$ and
exponents in $\Q$, which we also simply denote by $k((\Q))$.

\begin{theorem}                             \label{MT4}
Let $\Q^{\rm rc}$ denote the field of real algebraic numbers and
$\PSF(\Q^{\rm rc})$ the Puiseux series field over $\Q^{\rm rc}$. If $I$
is any integer part of this real closed field, then $\Quot I$ is a
proper countable subfield of $\PSF(\Q^{\rm rc})$ such that the
transcendence degree of $\PSF(\Q^{\rm rc})$ over $\Quot I$ is
uncountable. The same holds for the completion of $\PSF(\Q^{\rm rc})$.
\end{theorem}

This answers a question of M.~Moniri. An answer was also given,
independently, by L.~van den Dries at the conference. A larger variety
of such fields is presented in Section~\ref{sectabsir}. On the other
hand, there are fields that admit integer parts whose quotient field is
the whole field:

\begin{theorem}                             \label{MT5}
Let $\lambda$ be any cardinal number and $k$ any field of characteristic
$0$. Then there exists a henselian valued field $(L,v)$ with residue
field $k$ which has the following properties:
\sn
a) \ $L$ contains a $k$-algebra $R$ which is an additive complement of
its valuation ring such that $\Quot R=L$.\n
b) \ At the same time, for each non-zero cardinal number $\kappa \leq
\lambda$, $L$ contains a $k$-algebra $R_\kappa$ which is an additive
complement of its valuation ring such that $\trdeg L|\Quot
R_\kappa=\kappa$.

If in addition $k$ is an archimedean ordered field and $<$ is any
ordering on $L$ compatible with $v$ (see Section~\ref{sectprel} for this
notion), then $(L,<)$ admits an integer part $I$ such that $\Quot I=L$.
At the same time, for each non-zero cardinal number $\kappa \leq
\lambda$, $L$ admits an integer part $I_\kappa$ such that $\trdeg
L|\Quot I_n=\kappa$.
\end{theorem}

\pars
S.~Boughattas [Bg] has given an example of an ordered (and ``$n$-real
closed'') field which does not admit any integer part. In the last
section of our paper, we generalize the approach and consider a notion
that comprises integer parts as well as subrings which are additive
complements of the valuation ring or of the valuation ideal in a valued
field. A subring $R$ of a valued field $(K,v)$ will be called a
\bfind{weak complement} (in $K$) if it has the following properties:
\sn
$\bullet$ \ $vr\leq 0$ for all $r\in R$,\n
$\bullet$ \ for all $a\in K$ there is $r\in R$ such that $v(a-r)\geq 0$.
\sn
Every integer part in a non-archimedean ordered field $K$ is a weak
complement with respect to the natural valuation of $K$ (see
Lemma~\ref{ipwc}).

Using a somewhat surprising little observation (Lemma~\ref{embdsrf})
together with a result of [K1] (which is a generalization of a result in
[M--S]) we construct examples for valued fields that do not admit any
weak complements. From this we obtain ordered fields without integer
parts. In particular, we show:

\begin{theorem}                             \label{MT3}
For every prime field $k$ there are valued rational function fields
$k(t,x,y)$ of transcendence degree $3$ over the trivially valued
subfield $k$ which do not admit any weak complements. There are ordered
rational function fields of transcendence degree $3$ over $\Q$ which do
not admit any integer parts.

There are valued rational function fields of transcendence degree
$4$ over a trivially valued prime field which do not admit any weak
complements, but admit an embedding of their residue field and a
cross-section. There are ordered rational function fields of
transcendence degree $4$ over $\Q$ which do not admit any integer parts,
but admit an embedding of their residue field and a cross-section for
their natural valuation.
\end{theorem}

Our example of an $n$-real closed field without integer parts is the
$n$-real closure of such an ordered rational function field. It is quite
similar to the example given by Boughattas, but in contrast to his
example, ours is of finite transcendence degree over $\Q$.

\sn
{\bf Open Problem:} \
Are there valued fields of transcendence degree $\leq 2$ over a
trivially valued ground field that do not admit any weak complements?
Are there ordered fields of transcendence degree $\leq 2$ over an
archimedean ordered field that do not admit any integer parts? Are there
examples of transcendence degree $\leq 3$ with embedding of their
residue field and cross-section?

%
%
\section{Some preliminaries}                \label{sectprel}
For basic facts from general valuation theory we refer the reader to
[E], [R], [W], [Z--S], [K2]. For ramification theory, see [N], [E] and
[K2]. In the following, we state some well known facts without proofs.

\parm
Take any valued field $(K,v)$. If $v'$ is a valuation on the residue
field $Kv$, then $v\circ v'$ will denote the valuation whose valuation
ring is the subring of the valuation ring of $v$ consisting of all
elements whose $v$-residue lies in the valuation ring of $v'$. (Note
that we identify equivalent valuations.) While $v\circ v'$ does actually
not mean the composition of $v$ and $v'$ as mappings, this notation is
used because in fact, up to equivalence the place associated with
$v\circ v'$ is indeed the composition of the places associated with $v$
and $v'$.

Every convex subgroup $\Gamma$ of $vK$ gives rise to a coarsening
$v_\Gamma$ of $v$ such that $v_\Gamma K$ is isomorphic to $vK/\Gamma$.
As mentioned in the introduction, $v$ induces a valuation
$\ovl{v}_\Gamma$ on the residue field $Kv_\Gamma\,$. We then have that
$v=v_\Gamma\circ\ovl{v}_\Gamma\,$. The value group $\ovl{v}_\Gamma
(Kv_\Gamma)$ of $\ovl{v}_\Gamma$ is isomorphic to $\Gamma$, and its
residue field $(Kv_\Gamma) \ovl{v}_\Gamma$ is isomorphic to $Kv$. Every
coarsening $w$ of $v$ is of the form $v_\Gamma$ for some convex subgroup
$\Gamma$ of $vK$.

If $a$ is an element of the valuation ring ${\cal O}_v$ of $v$ on $K$,
then $av$ will denote the image of $a$ under the residue map associated
with the valuation $v$. This map is a ring homomorphism from ${\cal
O}_v$ onto the residue field $Kv$. It is only unique up to equivalence,
i.e., up to composition with an isomorphism from $Kv$ to another field
(and so the residue field $Kv$ is only unique up to isomorphism).
If $w$ is a coarsening of $v$, that is, ${\cal O}_v$ contains the
valuation ring ${\cal O}_w$ of $v$ on $L$, then the residue map
${\cal O}_w\ni a\mapsto aw\in Kw$ can be chosen such that it extends the
residue map ${\cal O}_v\ni a\mapsto av\in Kv$.

\parm
An ordering $<$ on a valued field $(K,v)$ is said to be
\bfind{compatible with the valuation} $v$ (and $v$ is
\bfind{compatible with } $<$) if
\begin{equation}                            \label{v<comp}
\forall x,y\in K:\;\; 0<x\leq y\,\Rightarrow\,vx\geq vy\;.
\end{equation}
This holds if and only if the valuation ring of $v$ is a convex subset
of $(K,<)$. This in turn holds if and only if $<$ induces an ordering on
the residue field $Kv$. We will need the following well-known facts
(cf.\ [P]):
\begin{lemma}                               \label{liftord}
Take any valued field $(K,v)$. Every ordering $<_r$ on $Kv$ can be
lifted to an ordering $<$ on $K$ which is compatible with $v$ and
induces $<_r$ on $Kv$ (that is, if $a,b$ are elements of the valuation
ring of $v$ such that $a<b$, then $av=bv$ or $av<_r bv$).
\end{lemma}

\begin{lemma}                               \label{extcv}
If an ordering of a field $K$ is compatible with the valuation $v$ of
$K$, then $v$ extends to a valuation of the real closure $K^{\rm rc}$ of
$(K,<)$, which is still compatible with the ordering on $K^{\rm rc}$.
This extension is henselian, its value group $vK^{\rm rc}$ is the
divisible hull of $vK$, and its residue field $K^{\rm rc}v$ is the real
closure of $Kv$ (with respect to the induced ordering on $Kv$).
\end{lemma}

A compatible valuation of an ordered field $(K,<)$ is called the
\bfind{natural valuation} of $(K,<)$ if its residue field is archimedean
ordered. The natural valuation is uniquely determined, and every
compatible valuation is a coarsening of the natural valuation.

\parm
Take any valued field $(K,v)$ and a finite extension $L|K$. Then the
following {\bf fundamental inequality} holds:
\begin{equation}                            \label{fundineq}
n\;\geq\;\sum_{i=1}^{\rm g} {\rm e}_i {\rm f}_i\;,
\end{equation}
where $n=[L:K]$ is the degree of the extension, $v_1,\ldots,v_{\rm g}$
are the distinct extensions of $v$ from $K$ to $L$, ${\rm e}_i=
(v_iL:vK)$ are the respective ramification indices, and ${\rm f}_i=
[Lv_i:Kv]$ are the respective inertia degrees. Note that the
extension of $v$ from $K$ to $L$ is unique (i.e., ${\rm g}=1$) if and
only if $(K,v)$ is henselian (which by definition means that $(K,v)$
satisfies Hensel's Lemma). The following are easy consequences:

\begin{lemma}                               \label{degvr}
If $L|K$ is a finite extension and $v$ is a valuation on $L$, then
$[L:K]\geq (vL:vK)$ and $[L:K]\geq [Lv:Kv]$.
\end{lemma}

\begin{corollary}                           \label{aaa}
Let $L|K$ be an algebraic extension and $v$ a valuation on $L$. Then
$vL/vK$ is a torsion group and the extension $Lv|Kv$ of residue fields
is algebraic. If $v$ is trivial on $K$ (i.e., $vK=\{0\}$), then $v$
is trivial on $L$.
\end{corollary}

An extension $(K,v)\subseteq (L,v)$ of valued fields is called
\bfind{immediate} if the canonical embeddings of $vK$ in $vL$ and of
$Kv$ in $Lv$ are onto. We have:
\begin{lemma}                               \label{denseimm}
If $K$ is dense in $(L,v)$, then $(K,v)\subseteq (L,v)$
is an immediate extension.
\end{lemma}
\begin{proof}
If $a\in L$ and $b\in K$ such that $v(a-b)>va$, then $va=vb\in vK$. If
$a\in L$ such that $va=0$ and $b\in K$ such that $v(a-b)>0$, then
$av=bv\in Kv$.
\end{proof}

The following is a well known consequence of the so-called ``Lemma of
Ostrowski'':

\begin{lemma}                               \label{lostr}
If a valued field $(L,v)$ is an immediate algebraic extension of a
henselian field $(K,v)$ of residue characteristic $0$, then $L=K$.
\end{lemma}

\begin{lemma}                               \label{henseliz}
The henselization $K^h$ of a valued field $(K,v)$ (which is unique up to
valuation preserving isomorphism over $K$) is an immediate extension and
can be chosen in every henselian valued extension field of $(K,v)$.
\end{lemma}

\begin{lemma}                               \label{algehens}
An algebraic extension of a henselian valued field, equipped with the
unique extension of the valuation, is again henselian.
\end{lemma}

\begin{lemma}                               \label{densevw}
Let $(L,v)$ be any field and $v=w\circ \ovl{w}$ where $w$ is
non-trivial. Take any subfield $L_0$ of $L$. Then $L_0$ is dense in
$(L,v)$ if and only if $L_0$ is dense in $(L,w)$.
\end{lemma}

\begin{lemma}                               \label{henswov}
Let $(K,v)$ be any field and $v=w\circ \ovl{w}$. Then
$(K,v)$ is henselian if and only if $(K,w)$ and $(Kw,\ovl{w})$ are.
\end{lemma}

\begin{corollary}                           \label{henswovcor}
Let $(K,v)$ be any field and $v=w\circ \ovl{w}$. If $(Kw,\ovl{w})$ is
henselian, then the henselization of $(K,v)$ is equal to the
henselization of $(K,w)$ (as fields).
\end{corollary}

The value group $vK$ of a valued field $(K,v)$ is \bfind{archimedean} if
it is embeddable in the ordered additive group of the reals. This holds
if and only if every convex subgroup of $vK$ is equal to $\{0\}$ or to
$vK$.

\begin{lemma}                               \label{rk1dense}
If $(K,v)$ is a valued field such that $vK$ is archimedean, then $K$
is dense in its henselization. In particular, the completion of
$(K,v)$ is henselian.
\end{lemma}

The following result is an easy application of Hensel's Lemma:
\begin{lemma}                               \label{for}
Take $(K,v)$ to be a henselian valued field of residue characteristic
$\chara Kv =0$. Take any subfield $K_0$ of $K$ on which $v$ is trivial.
Then there is a subfield $K'$ of $K$ containing $K_0$ and such that $v$
is trivial on $K'$ and the residue map associated with $v$
induces an isomorphism from $K'$ onto $Kv$. If $Kv|K_0v$ is algebraic,
then so is $K'|K_0\,$.
\end{lemma}
\n
A field $K'$ as in this lemma is called a \bfind{field of
representatives for the residue field} $Kv$.

\begin{proposition}                               \label{ffnothens}
a) \ Take a non-empty set $T$ of elements algebraically independent over
$K$ and a finite extension $F$ of $K(T)$. Then no non-trivial valuation
on $F$ is henselian. In particular, no non-trivial valuation on an
algebraic function field (of transcendence degree at least one) is
henselian.
\sn
b) \ Fix $n\in\N$, take $K(T)$ as in a) and take $F$ to be the closure
of $K(T)$ under successive adjunction of roots of polynomials of degree
$\leq n$. Then no non-trivial valuation on $F$ is henselian.
\end{proposition}
\begin{proof}
Choosing any $t\in T$ and replacing $K$ by $K(T\setminus \{t\})$, we may
assume in parts a) and b) that $T$ consists of a single element, i.e.,
$\trdeg F|K=1$.

Take any non-trivial valuation on $F$. We show that there is some $x\in
K(T)$ such that $vx>0$ and $x$ is transcendental over $K$. Assume first
that $v$ is trivial on $K$. Since $v$ is non-trivial on $F$ and $F|K(T)$
is algebraic, Corollary~\ref{aaa} shows that $v$ is non-trivial on
$K(T)$. Hence there must be some $x\in K(T)$ such that $vx\ne 0$.
Replacing $x$ by $x^{-1}$ if necessary, we may assume that $vx>0$. It
follows that $x\notin K$, so $x$ is transcendental over $K$.

Now assume that $v$ is not trivial on $K$, and take an arbitrary $x\in
K(T)$ transcendental over $K$. If $vx>0$, we are done. If $vx<0$, we
replace it by $x^{-1}$ and we are done again. If $vx=0$, we pick some
$c\in K$ such that $vc>0$. Then $vcx>0$ and $cx$ is transcendental
over $K$, hence replacing $x$ by $cx$ finishes the proof of our claim.

Pick any positive integer $q$ such that $q$ is not divisible by the
characteristic $p:=\chara Kv$ of the residue field $Kv$. By Hensel's
Lemma, any henselian extension of $K(x)$ will contain a $q$-th root of
the $1$-unit $y:=1+x$. We wish to show that any algebraic extension of
$K(x)$ containing such a $q$-th root must be of degree at least $q$ over
$K(x)$. A valuation theoretical proof for this fact reads as follows.
Take the $y$-adic valuation $v_y$ on $K(x)=K(y)$. Then $v_y y$ is the
least positive element in the value group $v_y K(x)\isom\Z$, and any
$q$-th root $b$ of $y$ will have $v_y$-value $\frac{1}{q}v_y y$. This
shows that $(v_y K(x)(b):v_y K(x))\geq q$. By the fundamental
inequality, it follows that $[K(x,b):K(x)]\geq (v_y K(x,b): v_y
K(x))\geq q$.

\sn
Proof of part a): \ Since $\trdeg F|K=1$, $x\in K(T)$ is transcendental
over $K$ and $F|K$ is finite, also $F|K(x)$ is finite. Pick $q>[F:K(x)]$
not divisible by $p$. Then it follows that $F$ does not contain a $q$-th
root of $y$, and so $(F,v)$ cannot be henselian.

\sn
Proof of part b): \ This time, we still have that $K(T)|K(x)$ is finite.
Pick a prime $q>\max\{n,[K(T):K(x)]\}$, $q\ne p$. For every element
$\alpha$ in the value group $v_y F$ there is an integer $e$
which is a product of positive integers $\leq n$ such that $e\alpha\in
v_y K(T)$. Further, there is a positive integer $e'$ such that
$e'e\alpha\in v_y K(x)$. On the other hand, by our choice of $q$, it
does not divide $e'e$. Since the order of the value $\frac{1}{q}v_y y$
modulo $v_y K(x)$ is $q$, it follows that this value does not lie in
$v_y F$. Hence again, $(F,v)$ cannot be henselian.
\end{proof}

\begin{proposition}                               \label{excomplh}
Take $(L,v)$ to be a henselian field of residue characteristic $\chara
Lv=0$, and $K$ a subfield of $L$ such that $L|K$ is algebraic. Then $K$
admits an algebraic extension $L_0$ inside of $L$ such that the
extension of $v$ from $K$ to $L_0$ is unique, $L_0$ is linearly disjoint
over $K$ from the henselization $K^h$ of $K$ in $L$, and
$L=L_0.K^h=L_0^h$.
\end{proposition}
\begin{proof}
Take any subextension $L_0|K$ of $L|K$ maximal with the property that
the extension of $v$ from $K$ to $L_0$ is unique. By general
ramification theory it follows that $L_0|K$ is linearly disjoint from
$K^h|K$ and that $L_0^h=L_0.K^h$. We only have to show that $L_0^h=L$.
Note that $L|L_0^h$ is algebraic since already $L|K$ is algebraic,

Let us show that $L_0v=Lv$. If this is not the case, then there is
be some element $\zeta\in Lv\setminus L_0v$. By Corollary~\ref{aaa},
$Lv|L_0v$ is algebraic. Let $g\in L_0v[X]$ be the minimal polynomial of
$\zeta$ over $L_0v$. Since $\chara Kv=0$, $g$ is separable. We choose
some monic polynomial $f$ with integral coefficients in $L_0$ whose
reduction modulo $v$ is $g$; it follows that $\deg f=\deg g$. Since
$\zeta$ is a simple root of $g$, it follows from Hensel's Lemma that the
henselian field $(L,v)$ contains a root $z$ of $f$ whose residue is
$\zeta$. We have
\[[L_0(z):L_0]\>\leq\deg f \>=\>\deg g \>=\>[L_0v(\zeta):L_0v]\>\leq\>
[L_0(z)v:L_0v]\>\leq\>[L_0(z):L_0]\;,\]
where the last inequality follows from Lemma~\ref{degvr}. We conclude
that $[L_0(z):L_0]=[L_0(z)v:L_0v]$. From the fundamental inequality it
follows that the extension of $v$ from $L_0$ (and hence also from $K$)
to $L_0(z)$ is unique. But this contradicts the maximality of $L_0\,$.
Hence, $L_0v=Lv$.

Next, let us show that $vL_0=vL$. If this is not the case, then there is
some $\alpha\in vL\setminus vL_0\,$. By Corollary~\ref{aaa}, $vL/vL_0$
is a torsion group and hence there is some $n>1$ such that $n\alpha\in
vL_0\,$. We choose $n$ minimal with this property, so that
$(vL_0+\alpha\Z:vL_0)=n$. Further, we pick some $a\in L$ such that $va=
\alpha$. Since $n\alpha\in vL_0\,$, there is some $d\in L_0$ such that
$vd=n\alpha=va^n$. It follows that $va^n/d=0$, and since we have
already shown that $Lv=L_0v$, we can choose some $c\in L_0$ such that
$(a^n/cd)v=1$. Consequently, the reduction of $X^n-a^n/cd$ modulo $v$ is
the polynomial $X^n-1$, which admits $1$ as a simple root since
$\chara Kv=0$. Hence by Hensel's Lemma, $X^n-a^n/cd$ admits a root $b$
in the henselian field $(L,v)$. For $z:=\frac{a}{b}$ it follows that
\[nvz\>=\>v\frac{a^n}{b^n}\>=\>vcd\>=\>vd\>=\>n\alpha,\]
which shows that $\alpha=vz\in vL_0(z)$. We have
\[[L_0(z):L_0]\>\leq \>n \>=\>(vL_0+\alpha\Z:vL_0)\>\leq\>
(vL_0(z):vL_0)\>\leq\>[L_0(z):L_0]\;,\]
where again the last inequality follows from Lemma~\ref{degvr}. We
conclude that $[L_0(z):L_0]=(vL_0(z):vL_0)$. From the fundamental
inequality it follows that the extension of $v$ from $L_0$ (and hence
also from $K$) to $L_0(z)$ is unique. But this again contradicts the
maximality of $L_0\,$. Hence, $vL_0=vL$.

We have shown that $vL=vL_0$ and $Lv=L_0v$. Hence, $vL=vL_0^h$ and
$Lv=L_0^hv$. As $L|L_0$ is algebraic, the same is true for $L|L_0^h$.
Since the residue field characteristic of $(L,v)$ is zero,
Lemma~\ref{lostr} shows that $L=L_0^h$. This concludes our proof.
\end{proof}

%
%
\section{Dense subfields}                   \label{sectds}
In this section we prove the existence of proper dense subfields of
henselian fields with residue characteristic $0$.

\begin{proposition}                         \label{d1}
Take a henselian valued field $(L,v)$ such that $vL$ admits a maximal
proper convex subgroup $\Gamma$. Assume that $\chara Lv_\Gamma=0$.
Then $L$ admits a proper dense subfield $L_0$ such that $L|L_0$ is
algebraic.
\end{proposition}
\begin{proof}
By Lemma~\ref{densevw} it suffices to find a subfield $L_0$ which is
dense in $L$ with respect to $v_\Gamma\,$, and such that $L|L_0$ is
algebraic. By Lemma~\ref{henswov}, $(L,v_\Gamma)$ is henselian. Since
$\chara Lv_\Gamma=0$ and hence $\chara L=0$, $L$ contains $\Q$ and
$v_\Gamma$ is trivial on $\Q$. Pick a transcendence basis $T$ of $L|\Q$.
Since $v_\Gamma$ is non-trivial on $L$, $T$ is non-empty. We infer from
Lemma~\ref{ffnothens} that $(\Q(T),v_\Gamma)$ is not henselian. By
Proposition~\ref{excomplh}, there is an algebraic extension $L_0$ of
$\Q(T)$ within $L$ such that $L_0$ is linearly disjoint over $\Q(T)$
from the $v_\Gamma$-henselization $\Q(T)^h$ of $\Q(T)$ in $L$, and
$L=L_0.\Q(T)^h= L_0^h$. Since $(\Q(T),v_\Gamma)$ is not henselian,
$\Q(T)^h|\Q(T)$ is a proper extension. By the linear disjointness, the
same holds for $L|L_0\,$. As $\Gamma$ is the maximal proper convex
subgroup of $vL$, $v_\Gamma L \isom vL/\Gamma$ must be archimedean. Thus
by Lemma~\ref{rk1dense}, $(L_0,v_\Gamma)$ lies dense in its
henselization $(L,v_\Gamma)$. Hence by Lemma~\ref{densevw}, $(L_0,v)$
lies dense in its henselization $(L,v)$. Since $L|\Q(T)$ is algebraic,
so is $L|L_0\,$.
\end{proof}

In certain cases, even if $v$ has a coarsest non-trivial coarsening,
there will also be dense subfields $K$ such that $L|K$ is
transcendental. For instance, this is the case for $L=k((t))$ equipped
with the $t$-adic valuation $v_t\,$, where a subfield is dense in $L$ as
soon as it contains $k(t)$. On the other hand, the henselization
$k(t)^h$ of $k(t)$ w.r.t.\ $v_t$ admits $k(t)$ as a proper dense
subfield, and the extension $k(t)^h|K$ is algebraic for every subfield
$K$ which is dense in $k(t)^h$. More generally, the following holds:

\begin{proposition}
Suppose that $(L,v)$ is a valued field and that $v$ is trivial on the
prime field $k$ of $L$. If
\[\trdeg L|k\>=\>\dim_\Q(\Q\otimes vL)\>+\>\trdeg Lv|k\><\>\infty\;,\]
then $L|K$ is algebraic for every dense subfield $K$.
\end{proposition}
\begin{proof}
If $K$ is a dense subfield, then by Lemma~\ref{denseimm}, $(L|K,v)$ is
an immediate extension. Hence,
\[\trdeg K|k\geq\dim_\Q (\Q\otimes vK)\>+\>\trdeg
Kv|k=\dim_\Q (\Q\otimes vL)\>+\>\trdeg Lv|k =\trdeg L|k\;,\]
whence $\trdeg K|k=\trdeg L|k$, showing that $L|K$ is algebraic.
\end{proof}
\n
Note that if $(L,v)$ is a valued field with a subfield $L_0$ on which
$v$ is trivial, and if $\trdeg L|L_0<\infty$, then in general,
\begin{equation}                            \label{abh}
\trdeg L|L_0\>\geq\>\dim_\Q(\Q\otimes vL)\>+\>\trdeg Lv|L_0\;.
\end{equation}
This is a special case of the so-called ``Abhyankar inequality''. For
a proof, see [Br], Chapter VI, \S10.3, Theorem~1. Note that $\Q\otimes
vL$ is the divisible hull of $vL$, and $\dim_\Q(\Q\otimes vL)$ is the
maximal number of rationally independent elements in $vL$.

\begin{remark}                              \label{remtce}
It can be shown that if $\chara Lv=0$, then the dense subfield $L_0$ in
Proposition~\ref{d1} can always be constructed in such a way that it
admits a truncation closed embedding into a power series field. The idea
is as follows. Since $(L,v)$ is henselian, we can use Lemma~\ref{for} to
find a field $k$ of representatives in $L$ for the residue field $Lv$.
Then we can choose a twisted cross-section as in [F]. The field $L_1$
generated over $k$ by the image of the cross-section admits a
truncation closed embedding in $k((vL))$ with a suitable factor system,
and this embedding $\iota$ can be extended to a truncation closed
embedding of $(L,v)$ in $k((vL))$ (cf.\ [F]). It is easy to show that
$L_\Gamma:=\iota^{-1}(\,\iota L\cap k((\Gamma))\,)$ is a field of
representatives for the residue field $Lv_\Gamma$ in $(L,v_\Gamma)$,
and that $\iota$ induces a truncation closed embedding of $L_\Gamma$ in
$k((\Gamma))\subset k((vL))$. This can be extended to a truncation
closed embedding of $L_2:=L_1.L_\Gamma$ which is obtained from
$L_\Gamma$ by adjoining the image of the cross-section. We note that
$(L,v_\Gamma)$ is an immediate extension of $(L_2,v_\Gamma)$. If this
extension is algebraic, then $L$ is also algebraic over the
henselization of $L_2$ (with respect to $v_\Gamma$), and by
Lemma~\ref{lostr}, the two fields must be equal. That shows that $L_2$
is dense in $(L,v_\Gamma)$ and hence in $(L,v)$, and we can take
$L_0=L_2\,$.

If $L|L_2$ is transcendental, we take a transcendence basis $S$ of
$L|L_2$ and pick $s\in S$. Then one shows as before that $L$ is the
henselization of $L_2(S)$, and also of the larger field $L_0:=
L_2(S\setminus \{s\})^h(s)$. Again, $L_0$ is dense in $(L,v)$. Following
[F], $L_2(S\setminus \{s\})^h$ admits a truncation closed embedding in
$k((vL))$. As $(L,v)$ is immediate over $(L_2,v)$, it is also immediate
over $L_2(S\setminus \{s\})^h$. Therefore, $s$ is the limit of a pseudo
Cauchy sequence in $L_2(S\setminus \{s\})^h$ without a limit in this
field. As the field is henselian of residue characteristic $0$, this
pseudo Cauchy sequence is of transcendental type. Now [F] shows that the
truncation closed embedding can be extended to $L_0\,$.
\end{remark}

\parb
Now we turn to the case where $vL$ admits no maximal proper convex
subgroup, i.e., $v$ admits no coarsest non-trivial coarsening. Such
valued fields exist:
\begin{lemma}                               \label{exvgcof}
Take any regular cardinal number $\lambda$ and any field $k$. Then there
is a valued field $(L,v)$ with residue field $k$ and such that $\lambda$
is the cofinality of the set of all proper convex subgroups of $vL$,
ordered by inclusion.
\end{lemma}
\begin{proof}
Take $J$ to be the set of all ordinal numbers $<\lambda$, endowed
with the reverse of the usual ordering. Choose any archimedean ordered
abelian group $\Gamma$. Then take $G$ to be the ordered Hahn product
${\bf H}_J \Gamma$ with index set $J$ and components $\Gamma$ (see [Fu]
or [KS] for details on Hahn products). Then the set of all proper convex
subgroups of $G$, ordered by inclusion, has order type $\lambda$
and hence has cofinality $\lambda$. Now take $(L,v)$ to be the power
series field $k((G))$ with its canonical valuation.
\end{proof}

Note that if $vL$ admits no maximal proper convex subgroup, then $vL$ is
the union of its proper convex subgroups. Indeed, if $\alpha\in vL$,
then the smallest convex subgroup $C$ of $vL$ that contains $\alpha$
($=$ the intersection of all convex subgroups containing $\alpha$)
admits a largest convex subgroup, namely the largest convex subgroup of
$vL$ that does not contain $\alpha$ ($=$ the union of all convex
subgroups not containing $\alpha$). Therefore $C\ne vL$, showing that
$C$ is a proper convex subgroup containing $\alpha$.

\begin{proposition}                         \label{d2}
Take a henselian valued field $(L,v)$ such that $vL$ admits no maximal
proper convex subgroup. Assume that $\chara L=0$. Then $L$ admits a
proper dense subfield $K$ such that $L|K$ is algebraic. If $\kappa>0$ is
any cardinal number smaller than or equal to the cofinality of the set
of convex subgroups of $vL$ ordered by inclusion, then there is also a
henselian (as well as a non-henselian) subfield $K$ dense in $L$ such
that $\trdeg L|K=\kappa$.
\end{proposition}
\begin{proof}
It suffices to prove that there is a subfield $K$ dense in $L$ such that
the trans\-cendence degree of $L|K$ is equal to the cofinality $\lambda$
of the set of convex subgroups of $vL$. This is seen as follows. Take a
transcendence basis $T$ of $L|K$. If $\kappa$ is a cardinal number
$\leq \lambda$, then take a subset $T_\kappa$ of $T$ of cardinality
$\kappa$. Then $K_\kappa:=K(T\setminus T_\kappa)$ is dense in $L$
because it contains $K$; furthermore, $\trdeg L|K_\kappa=\kappa$. We may
always, even in the case of $\kappa= \lambda$, choose $T_\kappa\ne T$.
Then by part a) of Proposition~\ref{ffnothens}, $(K_\kappa,v)$ is not
henselian. In particular, $(K(T),v)$ is not henselian and thus, $K(T)$
is a proper subfield of $L$ such that $L|K(T)$ is algebraic. If
$\kappa\ne 0$, then $L|K_\kappa$ will be transcendental. By
Lemma~\ref{henseliz}, the henselian field $L$ contains the henselization
$K^h_\kappa$ of $K_\kappa$. Since it is an algebraic extension of
$K_\kappa\,$, we have $\trdeg L|K^h_\kappa=\trdeg L|K_\kappa=\kappa$,
and it is dense in $L$, too.

To illustrate the idea of our proof, we first show that there is a dense
subfield $K$ such that $\trdeg L|K>0$. We choose a convex subgroup $C_0$
of $vL$ as follows. If $\chara Lv=0$, then we set $C_0=\{0\}$. If
$\chara Lv=p>0$, then we observe that $0\ne p\in L$ since $\chara L=0$,
so we may take $C_0$ to be the smallest proper convex subgroup that
contains $vp$. We let $w_0=v_{C_0}$ be the coarsening of $v$ associated
with $C_0\,$. We have $w_0=v$ if $\chara Lv=0$. Since $C_0$ is a proper
convex subgroup, $w_0$ is a non-trivial valuation.

Let $\lambda$ be the cofinality of the set of all proper convex
subgroups of $vL$, ordered by inclusion. Starting from $C_0\,$, we pick
a strictly ascending cofinal sequence of convex subgroups $C_{\nu}$,
$\nu<\lambda$, in this set. We denote by $w_{\nu}$ the coarsening of $v$
which corresponds to $C_\nu\,$.

By Lemma~\ref{for} there is a field $K'_0$ of representatives for
$Lw_0$ in $L$. We pick a transcendence basis $T_0=\{t_{0,\mu}\mid \mu<
\kappa_0\}$ of $K'_0|\Q$, where $\kappa_0$ is the transcendence degree
of $Lw_0|\Q$.
Then we proceed by induction on $\nu<\lambda$. Suppose we have already
constructed a field $K'_{\nu}$ of representatives of $Lw_{\nu}\,$ and a
transcendence basis $\bigcup_{\nu'\leq\nu}T_{\nu'}$ for it. By
Lemma~\ref{for}, $K'_{\nu}$ can be extended to a field $K'_{\nu+1}$ of
representatives of $Lw_{\nu+1}\,$, and we choose a transcendence basis
$T_{\nu+1}=\{t_{\nu+1,\mu}\mid \mu<\kappa_{\nu+1}\}$ of $K'_{\nu+1}|
K'_{\nu}$.
Having constructed $K'_{\nu}$, $\nu<\lambda'$ for some limit
ordinal $\lambda'\leq\lambda$, we set $K^*_{\lambda'}=
\bigcup_{\nu<\lambda'} K'_{\nu}\,$. Again by Lemma~\ref{for},
$K^*_{\lambda'}$ can be extended to a field of representatives
$K'_{\lambda'}$ of $Lw_{\lambda'}\,$, and we choose a transcendence
basis $T_{\lambda'}=\{t_{\lambda',\mu}\mid\mu<\kappa_{\lambda'}\}$ of
$K'_{\lambda'}|K^*_{\lambda'}$. Note that $T_{\lambda'}$ may be empty.

We set $K'=\bigcup_{\nu<\lambda} K'_{\nu}$ and show that $K'$ is dense
in $L$. Take any $a\in L$ and $\alpha\in vL$. Then there is some
$\nu<\lambda$ such that $\alpha\in C_{\nu}\,$. By construction, $K'$
contains a field of representatives for $Lw_{\nu}\,$. Hence there is
some $b\in K'$ such that $aw_{\nu}=bw_{\nu}\,$, meaning that
$w_{\nu}(a-b)>0$ and thus, $v(a-b)>\alpha$. This proves our claim.
Hence if $\trdeg L|K'>0$, we set $K=K'$ and we are done showing the
existence of a subfield $K$ with $\trdeg L|K>0$. But it may well
happen that $L|K'$ is algebraic, or even that $L=K'$. In this case, we
construct a subfield $K$ of $K'$ as follows.

Note that for all $\nu<\lambda$, $(K'_\nu,v)$ is henselian. Indeed, it
is isomorphic (by the place associated with $w_\nu$) to
$(Lw_{\nu},\ovl{w}_\nu)$, where $\ovl{w}_\nu$ is the valuation induced
by $v$ on $Lw_{\nu}\,$; since $(L,v)$ is henselian, Lemma~\ref{henswov}
shows that the same is true for $(Lw_{\nu},\ovl{w}_\nu)$ and hence for
$(K'_\nu,v)$. Again from Lemma~\ref{henswov} it follows that $(K'_\nu,
w_{\mu})$ is henselian for all $\mu<\lambda$. Note that $w_{\mu}$ is
non-trivial on $K'_\nu$ only for $\mu<\nu$, and in this case, $K'_\nu
w_\mu=Lw_\mu$ since $K'_\nu$ contains the field $K'_\mu$ of
representatives for $Lw_\mu\,$.

Note further that for all $\nu<\lambda$ and all $\mu<\kappa_\nu\,$,
$w_\nu t_{\nu,\mu}=0$. On the other hand, after multiplication with
suitable elements in $K_{\nu+1}$ we may assume that $w_{\nu}
t_{\nu+1,\mu}>0$ for all $\mu<\kappa_{\nu+1}\,$.

We will now construct inside of $K'$ a chain (ordered by inclusion) of
subfields $K_{\nu}\subset K'_{\nu}$ ($\nu<\lambda$) such that each
$K_{\nu}$ is a field of representatives for $Lw_{\nu}$ and contains the
element $t_{0,0}-t_{\nu+1,0}$, but not the element $t_{0,0}\,$.

Since $T_0=\{t_{0,\mu}\mid \mu< \kappa_0\}$ is a transcendence basis of
$K'_0|\Q$, Lemma~\ref{aaa} shows that the residue field $K'_1w_0=
K'_0w_0$ is algebraic over $\Q(t_{0,\mu}\mid \mu<\kappa_0)w_0\,$.
Because $(t_{0,0}-t_{1,0})w_0= t_{0,0}w_0$ by construction, the latter
field is equal to $\Q(t_{0,0}-t_{1,0},t_{0,\mu}\mid
1\leq\mu<\kappa_0)w_0\,$.
Since $\chara K'_1 w_0=0$, we can use Lemma~\ref{for} to find inside of
the henselian field $(K'_1,w_0)$ an algebraic extension $K_0$ of
$\Q(t_{0,0} -t_{1,0},t_{0,\mu}\mid 1\leq\mu<\kappa_0)$ which is a field
of representatives for $K'_1w_0=Lw_0\,$. Note that $t_{0,0}$ is
transcendental over $\Q(t_{0,0} -t_{1,0},t_{0,\mu}\mid
1\leq\mu<\kappa_0)$ and therefore, $t_{0,0}\notin K_0\,$, but
$t_{0,0} -t_{1,0}\in K_0\,$.

Suppose we have already constructed all fields $K_{\mu}$ for $\mu\leq
\nu$, where $\nu$ is some ordinal $<\lambda$. Since $T_{\nu+1}=
\{t_{\nu+1,\mu}\mid \mu<\kappa_{\nu+1}\}$ is a transcendence basis of
$K'_{\nu+1}|K'_\nu\,$, Lemma~\ref{aaa} shows that the residue field
$K'_{\nu+2}w_{\nu+1}=K'_{\nu+1}w_{\nu+1}$ is algebraic over $K_{\nu}
(t_{\nu+1,\mu}\mid \mu<\kappa_{\nu+1}) w_{\nu+1}\,$. Because
$(t_{\nu+1,0}-t_{\nu+2,0})w_{\nu+1}=t_{\nu+1,0} w_{\nu+1}$ by
construction, the latter field is equal to $K_{\nu} (t_{\nu+1,0}-
t_{\nu+2,0}, t_{\nu+1,\mu}\mid 1\leq\mu<\kappa_{\nu+1}) w_{\nu+1}\,$.
Since $\chara K'_{\nu+2} w_{\nu+1}=0$, we can use Lemma~\ref{for} to
find inside of the henselian field $(K'_{\nu+2},w_{\nu+1})$ an algebraic
extension $K_{\nu+1}$ of $K_{\nu}(t_{\nu+1,0}- t_{\nu+2,0},
t_{\nu+1,\mu} \mid 1\leq\mu<\kappa_{\nu+1})$ which is a field of
representatives for $K'_{\nu+2}w_{\nu+1}= Lw_{\nu+1}\,$. Since
$t_{0,0}-t_{\nu+1,0}\,,\, t_{\nu+1,0}-t_{\nu+2,0}\in K_{\nu+1}$ we have
that $t_{0,0}-t_{\nu+2,0} \in K_{\nu+1}\,$. Again, $t_{0,0} \notin
K_{\nu+1}$ as $t_{0,0}$ is transcendental over $K_{\nu} (t_{\nu+1,0} -
t_{\nu+2,0}, t_{\nu+1,\mu}\mid 1\leq\mu<\kappa_{\nu+1})$.

Suppose we have already constructed all fields $K_{\nu}$ for
$\nu< \lambda'$, where $\lambda'$ is some limit ordinal $\leq\lambda$.
We note that $t_{0,0}\notin \bigcup_{\nu<\lambda'}^{} K_{\nu}=:
K_{\lambda'}^{**}$. But $K_{\lambda'}^{**}(t_{0,0})$ contains the entire
transcendence basis of $K_{\lambda'}^*|\Q$ because $t_{0,0}-t_{\nu+1,0}
\in K_{\nu}$ for every $\nu<\lambda'$ (recall that $K_{\lambda'}^*$ is
the field we constructed above before constructing $K'_{\lambda'}$). It
follows that $T_{\lambda'}\cup \{t_{0,0}\}$ is a transcendence basis of
$K'_{\lambda'}| K^{**}_{\lambda'}$, and therefore the residue field
$K'_{\lambda'+1} w_{\lambda'}=Lw_{\lambda'}$ is an algebraic extension
of $K^{**}_{\lambda'} (T_{\lambda'}\cup\{t_{0,0}\})w_{\lambda'}\,$.
Because $(t_{0,0}-t_{\lambda'+1,0}) w_{\lambda'}= t_{0,0}w_{\lambda'}$,
the latter field is equal to $K^{**}_{\lambda'}(T_{\lambda'}\cup
\{t_{0,0} -t_{\lambda'+1,0}\}) w_{\lambda'}\,$. Again by
Lemma~\ref{for}, there is an algebraic extension $K_{\lambda'}$ of
$K^{**}_{\lambda'}(T_{\lambda'} \cup \{t_{0,0} -t_{\lambda'+1,0}\})$
inside of the henselian field $(K'_{\lambda'+1}, w_{\lambda'})$ which is
a field of representatives for $K'_{\lambda'+1}w_{\lambda'}=
Lw_{\lambda'}\,$. By construction, $t_{0,0} -t_{\lambda'+1,0}\in
K_{\lambda'}\,$. As before, $t_{0,0} \notin K_{\lambda'}$ as $t_{0,0}$
is transcendental over $K^{**}_{\lambda'}(T_{\lambda'} \cup \{t_{0,0}
-t_{\lambda'+1,0}\})$.

We set
\begin{equation}                            \label{K=U}
K\;:=\; \bigcup_{\nu<\lambda} K_{\nu}\;.
\end{equation}
By construction, $t_{0,0}\notin K$, but $K(t_{0,0})$ contains
$t_{\nu,\mu}$ for all $\nu<\lambda$ and $\mu< \kappa_\nu\,$. Hence,
$K'|K(t_{0,0})$ is algebraic and therefore, $\trdeg K'|K=1$. With the
same argument as for $K'$, one shows that $K$ is dense in $L$. (This
also follows from the fact that $t_{0,0}$ is limit of the Cauchy
sequence $(t_{0,0}-t_{\nu+1,0})_{\nu<\lambda}$ in $K$ and $K'$ is dense
in $L$.)

\parm
Now we indicate how to achieve $\trdeg K'|K=\lambda$. By passing to a
cofinal subsequence of $(C_\nu)_{\nu<\lambda}$ if necessary, we can
assume that every $T_{\nu}$ contains at least $|\nu|$ many elements,
where $|\nu|$ denotes the cardinality of the ordinal number $\nu$. Then
it is possible to re-order the elements of $T_\nu$ in such a way that
$T_\nu=\{t_{\nu,\mu}\mid \mu<\mu_\nu\}$ where $\mu_\nu$ is some ordinal
number $\geq\nu$. Now we modify the above construction of $K$ as
follows: at every step $\nu$ where $\nu=0$ or $\nu$ is a successor
ordinal, we replace $t_{\nu,\mu}$ by $t_{\nu,\mu}-t_{\nu+1,\mu}$ for all
$\mu\leq\nu$. In the limit case for $\lambda'<\lambda$, we then have
that $T_{\lambda'}\cup \{t_{\nu,\nu}\mid\nu<\lambda'\}$ is a
transcendence basis of $K'_{\lambda'}| K^{**}_{\lambda'}$. Here, we
replace every $t_{\nu,\nu}$ for $\nu<\lambda'$ by $t_{\nu,\nu}-
t_{\lambda'+1,\nu}$. In this way we achieve that the elements
$t_{\nu,\nu}\,$, $\nu<\lambda$ will be algebraically independent over
$K$, but $K$ will still be dense in $L$.
\end{proof}

\begin{remark}                              \label{r2}
We can replace the field $K_\kappa=K(T\setminus T_\kappa)$ mentioned in
the first paragraph of the proof by the larger field $K(T\setminus
T_\kappa \setminus \{t\})^h(t)$ where $t\in T\setminus T_\kappa$.
By the same argument as given at the end of Remark~\ref{remtce}, this
field admits a truncation closed embedding into the corresponding power
series field.
\end{remark}

\pars
Propositions~\ref{d1} and~\ref{d2} together prove Theorem~\ref{MT1}.
Theorem~\ref{MT2} follows immediately from Proposition~\ref{d2} since if
$vL$ admits no maximal proper convex subgroup, then the cofinality of
the set of convex subgroups of $vL$ is an infinite cardinal number. It
remains to give the
\sn
{\bf Proof of Proposition~\ref{cor1}:} \
By Lemma~\ref{exvgcof} we may take a henselian valued field $(L,v)$ of
residue characteristic $0$ such that $vL$ admits no maximal proper
convex subgroup. Using Proposition~\ref{d2} we pick a non-henselian
proper subfield $K$ which is dense in $L$. Lemma~\ref{densevw} shows
that for every non-trivial coarsening $w$ of $v$, $(K,w)$ is dense in
$(L,w)$, whence $Kw=Lw$. By Lemma~\ref{henswov}, $(Lw,\ovl{w})$ is
henselian because $(L,v)$ is henselian and $v=w\circ \ovl{w}$. Hence,
$(Kw,\ovl{w})$ is henselian, which finishes our proof.        \QED

\begin{example}                              \label{cenothens}
A more direct construction of a counterexample works as follows: Take an
ascending chain of convex subgroups $C_i\,$, $i\in\N$, in some ordered
abelian group. Take $k$ to be any field and set
\[K\>:=\>\bigcup_{i\in\N} k((C_i))\;.\]
As a union of an ascending chain of henselian valued fields, $K$ is
itself a henselian valued field. But $K$ is not complete. For
instance, if $0<\alpha_i\in C_i\setminus C_{i-1}\,$, then the element
\[x\>:=\>\sum_{i\in\N} t^{\alpha_i}\in k((\bigcup_{i\in\N} C_i))\]
lies in the completion of $K$, but not in $K$. Since every
henselian field is separable-algebraically closed in its henselization
(cf.\ [W], Theorem 32.19), $x$ is either transcendental or purely
inseparable over $K$. But it cannot be purely inseparable over $K$
because if $p=\chara K>0$, then $x^{p^\nu}=\sum_{i\in\N}
t^{p^\nu\alpha_i}\notin K$ for all $\nu\geq 0$. Hence by part a) of
Proposition~\ref{ffnothens}, $K(x)$ (endowed with the restriction $v$ of
the valuation of the completion of $K$) is not henselian. But for every
non-trivial coarsening $w$ of $v$, $K(x)w=Kw$ since $K$ is dense in
$(K(x),v)$, and we leave it to the reader to prove that $(Kw,\ovl{w})$
is henselian.
\end{example}

%
%
\section{Small integer parts}            \label{sectabsir}
We will use a cardinality argument to show that there are real
closed fields that are larger than the quotient fields of all its
integer parts.

\begin{lemma}                               \label{eqcard}
a) \ Take any valued field $(L,v)$. Then all additive complements
of the valuation ring of $L$, if there are any, have the same
cardinality.
\sn
b) \ All integer parts in an ordered field, if there are any, are
isomorphic as ordered sets and thus have the same cardinality.
\end{lemma}
\begin{proof}
a): \ As an additive group, any additive complement of the valuation
ring ${\cal O}$ of $L$ is isomorphic to $L/{\cal O}$.
\sn
b): \ Take two integer parts $I_1$ and $I_2$ of a given ordered field
$(L,<)$. Since $I_2$ is an integer part, for every $a\in I_1$ there is a
unique element $a'\in I_2$ such that $a'\leq a<a'+1$. Hence, we have a
mapping $I_1\ni a\mapsto a'\in I_2\,$. Conversely, since $I_1$ is an
integer part, there is a unique $a''\in I_1$ such that $a''< a' \leq
a''+1$. Consequently, $a=a''+1$ and $a$ is the only element that is
sent to $a'$, showing that the map is injective and even order
preserving. On the other hand, since $a''+1$ is sent to $a'$, the
mapping is also proved to be onto.
\end{proof}

We also need the following facts, which are well known (note that a
similar statement holds for weak complements):
\begin{lemma}                               \label{densacip}
a) If $K$ is dense in $(L,v)$, then every additive complement of the
valuation ring of $(K,v)$ is also an additive complement of the
valuation ring of $(L,v)$.
\sn
b) If $K$ is dense in $(L,<)$, then every integer part of $(K,<)$ is
also an integer part of $(L,<)$.
\end{lemma}
\begin{proof}
We only prove a) and leave the proof of b) to the reader. Let $A$ be an
additive complement of the valuation ring ${\cal O}_K$ of $(K,v)$, that
is, $A\cap {\cal O}_K=\{0\}$ and $A+{\cal O}_K=K$. Denote the valuation
ring of $(L,v)$ by ${\cal O}_L\,$. Since the valuation on $L$ is an
extension of the valuation on $K$, we have that $K\cap {\cal O}_L
={\cal O}_K$ and thus, $A\cap {\cal O}_L=A\cap {\cal O}_K=\{0\}$. Now
take any $a\in L$. Since $K$ is dense in $(L,v)$, there is $b\in K$ such
that $v(a-b)\geq 0$, that is, $a-b\in {\cal O}_L\,$. Consequently,
$a=b+(a-b)\in K+{\cal O}_L=A+{\cal O}_K+{\cal O}_L=A+{\cal O}_L\,$. This
proves that $A+{\cal O}_L=L$.
\end{proof}

\pars
We cite the following fact; for a proof, see for instance [B--K--K].
\begin{lemma}                               \label{c+z}
If $K$ is an ordered field and $R$ is a subring which is an additive
complement of the valuation ring for the natural valuation of $K$, then
$R+\Z$ is an integer part of $K$.
\end{lemma}

For every ordered abelian group $G$, written additively, we set
\[G^{<0}\;:=\;\{g\in G\mid g<0\}\;.\]

\begin{proposition}                         \label{propPSF}
Suppose that $k$ is a countable field. Then the countable ring
$k[\Q^{<0}]:=k[t^g\mid 0>g\in \Q]\subset k((\Q))$ is an additive
complement of the valuation ring of the uncountable henselian valued
field $\PSF(k)$. The same remains true if $\PSF(k)$ is replaced by its
completion.

If in addition $k$ is an ordered (respectively, real closed) field, then
$k[\Q^{<0}]+\Z$ is an integer part of the ordered (respectively, real
closed) field $\PSF(k)$, and this also remains true if $\PSF(k)$ is
replaced by its completion.
\end{proposition}
\begin{proof}
It is well known that every field $k((t))$ of formal Laurent series is
uncountable. Hence, $\PSF(k)$ is uncountable. As the union of an
ascending chain of fields $k((t^{\frac{1}{n}}))$ of formal Laurent
series, which are henselian, $\PSF(k)$ is itself henselian. Note that
the completion of a henselian field is again henselian ([W], Theorem
32.19).

Every element $a\in\PSF(k)$ lies in $k((t^{\frac{1}{n}}))$ for some
$n\in\N$. Hence, it suffices to show that $k[t^{\frac{m}{n}}\mid 0>m\in
\Z]$ is an additive complement of the valuation ring
$k[[t^{\frac{1}{n}}]]$ in $k((t^{\frac{1}{n}}))$. Renaming
$t^{\frac{1}{n}}$ by $t$, we thus have to show that $k[t^m\mid 0>m\in
\Z]$ is an additive complement of the valuation ring $k[[t]]$ in
$k((t))$. But this is clear since $k((t))$ is the set of formal Laurent
series
\[\sum_{i=N}^{\infty} c_i t^i \;=\; \sum_{i=N}^{-1} c_i t^i \>+\>
\sum_{i=0}^{\infty} c_i t^i\]
where $N\in\Z$ and $c_i\in k$. The first sum lies in $k[t^m\mid 0>m\in
\Z]$ and the second sum in $k[[t]]$.

Part a) of Lemma~\ref{densacip} shows that $k[\Q^{<0}]$ is also
an additive complement of the valuation ring in the completion of
$\PSF(k)$.

The assertions about the ordered case follow from Lemma~\ref{c+z}
together with part b) of Lemma~\ref{densacip}.
\end{proof}

From this proposition together with Lemma~\ref{eqcard}, we obtain the
following corollary, which in turn proves Theorem~\ref{MT4}.
\begin{corollary}                           \label{25}
Suppose that $k$ is a countable field. If $R$ is any subring which is an
additive complement of the valuation ring of $\PSF(k)$, then $\Quot R$
is countable and $\trdeg \PSF(k)|\Quot R$ is uncountable.

If in addition $k$ is an ordered field and $I$ an integer part of
$\PSF(k)$, then $\Quot I$ is countable and $\trdeg \PSF(k)|\Quot I$ is
uncountable.

The same remains true if $\PSF(k)$ is replaced by its completion.
\end{corollary}
\begin{proof}
The quotient field of a countable ring is again countable. So it only
remains to prove the assertion about the transcendence degree. It
follows from the fact that the algebraic closure of a countable field is
again countable. So if $T$ would be a countable transcendence basis of
$\PSF(k)|\Quot R$, then $(\Quot R)(T)$ and hence also $\PSF(k)$ would be
countable, which is not the case.
\end{proof}

Denote by $k((G))=k((t^G))$ the power series field with coefficients in
$k$ and exponents in $G$, and by $k(G)$ the smallest subfield of
$k((G))$ which contains all monomials $ct^g$, $c\in k$, $g\in G$.
Denote by $k(G)^c$ its completion; it can be chosen in $k((G))$.
Note that the completion of $\PSF(k)$ is equal to $k(\Q)^c$. Further,
denote by $k[G^{<0}]$ the subring of $k(G)$ generated by $k$ and all
monomials $ct^g$ where $c\in k$ and $0>g\in G$.

\begin{proposition}
Suppose that $k$ is a countable field and $G$ is a countable archimedean
ordered abelian group. Then the countable ring $k[G^{<0}]$ is an
additive complement of the valuation ring of the uncountable henselian
valued field $k(G)^c$.

If $R$ is any subring which is an additive complement of the valuation
ring of $k(G)^c$, then $\Quot R$ is countable and $\trdeg k(G)^c|\Quot
R$ is uncountable.

If in addition $k$ is an ordered field and $I$ is an integer part of
the ordered field $k(G)^c$, then $\Quot I$ is countable and $\trdeg
k(G)^c|\Quot I$ is uncountable.
\end{proposition}
\begin{proof}
By Lemma~\ref{rk1dense}, $k(G)^c$ is henselian. (Therefore, it is real
closed if and only if $k$ is real closed and $G$ is divisible.)

We show that the ring $k[G^{<0}]$ is an additive complement of the
valuation ring in $k(G)$. Every element $a$ of the latter is a quotient
of the form
\[a\;=\; \frac{c_1t^{g_1}+\ldots+c_m t^{g_m}}{d_1t^{h_1}+\ldots+d_n
t^{h_n}}\]
with $c_1,\ldots,c_m,d_1,\ldots,d_n\in k$ and $g_1,\ldots,g_m,h_1,\ldots,
h_n\in G$. Without loss of generality we may assume that $h_1$ is the
unique smallest element among the $h_1,\ldots,h_n\,$. Then we rewrite
$a$ as follows:
\[a\;=\; \frac{\frac{c_1}{d_1}t^{g_1-h_1}+\ldots+\frac{c_m}{d_1}
t^{g_m-h_1}}{1+\frac{d_2}{d_1}t^{h_2-h_1}+\ldots+\frac{d_n}{d_1}
t^{h_n-h_1}}\;.\]
By our assumption on $h_1\,$, all summands in the denominator except for
the $1$ have positive value. Hence, we can rewrite $a$ as
\[a\;=\; \left(\frac{c_1}{d_1}t^{g_1-h_1}+\ldots+\frac{c_m}{d_1}
t^{g_m-h_1}\right)\left(1+\sum_{i=1}^{\infty}(-1)^i\left(\frac{d_2}{d_1}
t^{h_2-h_1}+ \ldots+\frac{d_n}{d_1} t^{h_n-h_1}\right)^i\right)\;.\]
In the power series determined by this geometric series, only finitely
many summands will have negative value; this is true since $G$ is
archimedean by hypothesis. Let $b\in k[G^{<0}]$ be the sum of these
summands. Then $v(a-b)\geq 0$. This proves that $k[G^{<0}]$ is an
additive complement of the valuation ring in $k(G)$. Part a) of
Lemma~\ref{densacip} shows that $k[G^{<0}]$ is also an additive
complement of the valuation ring in $k(G)^c$.

All other assertions are deduced like the corresponding assertions of
Corollary~\ref{25}.
\end{proof}

%
%
\section{Proof of Theorem~\ref{MT5}}
We take $k$ to be any field of characteristic $0$ and
\[L\>:=\>\bigcup_{\nu<\lambda} k((C_\nu))\]
to be the henselian valued field constructed in Example~\ref{cenothens}.
The set Neg$\,k((C_\nu))$ of all power series in $k((C_\nu))$ with only
negative exponents is a $k$-algebra which is an additive complement of
the valuation ring $k[[C_\nu]]$ of $k((C_\nu))$. It follows that
\[R\;:=\;\bigcup_{\nu<\lambda} {\rm Neg}\,k((C_\nu))\]
is a $k$-algebra which is an additive complement of the valuation
ring $\bigcup_{\nu< \lambda} k[[C_\nu]]$ of $L$. We wish to show that
its quotient field is $L$. Take any element $a\in L$. Since $L$ is the
union of the $k((C_\nu))\,$, there is some $\nu$ such that $a\in
k((C_\nu))\,$. Pick some negative $\alpha\in C_{\nu+1}\setminus
C_\nu\,$. Then $\alpha<C_\nu\,$. Denote by $t^\alpha$ the monic monomial
of value $\alpha$ in $k((C_{\nu+1}))$. Then $at^\alpha$ has only
negative exponents, so $t^\alpha$ and $at^\alpha$ are both elements of
Neg$\,k((C_{\nu+1}))$. Therefore, $a\in\Quot {\rm Neg}\,(k((C_{\nu+1}))
\subseteq \Quot R$.

Now take any non-zero cardinal number $\kappa \leq \lambda$. We modify
the construction in the final part of the proof of Proposition~\ref{d2}
in that we start with $K'_\nu=k((C_\nu))$, and replace $t_{\nu,\nu}$ by
$t_{\nu,\nu}-t_{\nu+1,\nu}$ (or by $t_{\nu,\nu}- t_{\lambda'+1,\nu}$ in
the limit case) only as long as $\nu\leq \kappa$. Then the elements
$t_{\nu,\nu}\,$, $\nu<\kappa$, will be algebraically independent over
$K=\bigcup_{\nu<\lambda} K_\nu$, we have $\trdeg L|K=\kappa$, and $K$
will be dense in $L$.

For every $\nu<\lambda$, $w_\nu$ induces a valuation preserving
isomorphism from $(k((C_\nu)),v)$ and from $(K_\nu,v)$ onto
$(Lw_\nu,\ovl{w}_\nu)$. Hence, $\iota_{\nu}:=(w_\nu|_{K_\nu})^{-1}\circ
w_\nu|_{k((C_\nu))}$ is a valuation preserving isomorphism from
$(k((C_\nu)),v)$ onto $(K_\nu,v)$. For $\nu<\mu<\lambda$, $\iota_\mu$ is
an extension of $\iota_\nu\,$. Hence, $\iota:=\bigcup_{\nu<\lambda}
\iota_\nu$ is a valuation preserving isomorphism from $(L,v)$ onto
$(K,v)$. The image $R_\kappa$ of $R$ under $\iota$ is a $k$-algebra
which is an additive complement of the valuation ring of $K$ and has
quotient field $K$. Consequently, $\trdeg L|\Quot R_\kappa=\kappa$.
Since $K$ is dense in $L$, $R_\kappa$ is also an additive complement of
the valuation ring of $L$.

\pars
If in addition $k$ is an archimedean ordered field and $<$ is any
ordering on $L$ compatible with $v$, then $v$ is the natural valuation
of $(L,<)$. Hence by Lemma~\ref{c+z}, $R+\Z$ and $R_\kappa+\Z$ are
integer parts of $(L,<)$. Since $\Quot(R+\Z)=\Quot R$ and
$\Quot(R_\kappa+\Z)=\Quot R_\kappa\,$, this completes our proof.
                                                             \QED

%
%
\section{Weak complements}
\begin{lemma}                               \label{ipwc}
Let $I$ be an integer part of the ordered field $(K,<)$. If $v$ denotes
the natural valuation of $(K,<)$, then $I$ is a weak complement in
$(K,v)$.
\end{lemma}
\begin{proof}
Take $0<x\in K$ and assume that $vx>0$. Then for all $n\in\N$, also
$vnx>0=v1$ which by (\ref{v<comp}) implies that $0<nx\leq 1$.
Consequently, $0<x<1$ and thus, $x\notin I$. This proves that
$vr\leq 0$ for all $r\in I$.

For every $a\in K$ there is $r\in I$ such that $0\leq a-r<1$. Again by
(\ref{v<comp}), this implies that $v(a-r)\geq v1=0$.
\end{proof}

In what follows, let $R$ be a weak complement in a valued field $(K,v)$.
For every convex subgroup $\Gamma$ of $vK$, we define
\[R_\Gamma\>:=\>\{r\in R\mid vr\in\Gamma\cup\{\infty\}\}\;.\]

\begin{lemma}
For every convex subgroup $\Gamma$ of $vK$, $R_\Gamma$ is a
subring of $K$. Denote by $K_\Gamma$ its quotient field.
Then $vK_\Gamma=\Gamma$.
\end{lemma}
\begin{proof}
Take $r,s\in R_\Gamma\,$. Then $vr,vs\in \Gamma$. Since $r-s\in R$, we
have $0\geq v(r-s)\geq\min\{vr,vs\}$, showing that $v(r-s)\in \Gamma$
and thus $r-s\in R_\Gamma\,$. Further, $rs\in R$ and $vrs=vr+vs\in
\Gamma$, showing that $rs\in R_\Gamma\,$. This proves that $R_\Gamma$ is
a subring of $K$.

Since $vR_\Gamma:=\{vr\mid r\in R_\Gamma\}\subseteq\Gamma$, we know that
$vK_\Gamma\subseteq \{\alpha-\beta\mid\alpha,\beta\in\Gamma\}=\Gamma$.
On the other hand, for every $a\in K$ with $va \in \Gamma^{<0}$ there is
some $r\in R$ such that $v(a-r)\geq 0$. It follows that $vr=va\in\Gamma$
and thus $r\in R_\Gamma$ and $va=vr\in vR_\Gamma\,$. Hence,
$\Gamma^{<0}\subseteq vR_\Gamma\,$, which implies that
$vK_\Gamma=\Gamma$.
\end{proof}

Note that $K_\Gamma$ is a subfield of the quotient field of $R$. Since
$vK_\Gamma=\Gamma$, we have that $v_\Gamma K_\Gamma=\{0\}$. This means
that the residue map associated with $v_\Gamma$ induces an isomorphism
on $K_\Gamma\,$. This is in fact an isomorphism
\[(K_\Gamma, v)\>\isom\>(K_\Gamma v_\Gamma,\ovl{v}_\Gamma)\]
of valued fields.

\begin{lemma}                               \label{embdsrf}
For every non-trivial convex subgroup $\Gamma$ of $vK$, the valued
residue field $(K_\Gamma v_\Gamma,\ovl{v}_\Gamma)$ lies dense in
$(Kv_\Gamma,\ovl{v}_\Gamma)$.
\end{lemma}
\begin{proof}
We have to show: if $a\in K$ such that $va\in \Gamma$, then for every
positive $\gamma\in \Gamma$ such that $\gamma>va$ there is some $b\in
K_\Gamma$ such that $v(a-b) \geq\gamma$. Since $\Gamma^{<0}\subseteq
vR_\Gamma$ by the foregoing lemma, we may pick some $c\in R_\Gamma$ such
that $vc=-\gamma$. Then there is some $r\in R$ such that $v(ac-r)\geq
0$. Since $vac=va-\gamma\in\Gamma^{<0}$, we have $vr=vac\in\Gamma^{<0}$
and therefore, $r\in R_\Gamma$. Setting $b=\frac{r}{c}\in K_\Gamma\,$,
we obtain $v(a-b)\geq -vc=\gamma$.
\end{proof}

Now we give examples for valued fields and ordered fields without
weak complements or integer parts.
\sn
{\bf Basic construction:}
Take an arbitrary field $k$ and $t$ a transcendental element over $k$.
Denote by $v_t$ the $t$-adic valuation on $k(t)$. Choose some countably
generated separable-algebraic extension $(k_1,v_t)$ of $(k(t),v_t)$.
Take two algebraically independent elements $x,y$ over $k(t)$. Then by
Theorem~1.1 of [K1] there exists a non-trivial valuation $w$ on $K:=
k(t,x,y)$ whose restriction to $k(t)$ is trivial, whose value group is
$\Z$ and whose residue field is $k_1\,$; since $w$ is trivial on
$k(t)$, we may assume that the residue map associated with $w$ induces
the identity on $k(t)$. Now we take the valuation $v$ on the rational
function field $K$ to be the composition of $w$ with $v_t$:
\[v\>=\>w\circ v_t\;.\]

\begin{example} \rm                         \label{e1}
We take $k$ to be one of the prime fields $\Q$ or $\F_p$ for some prime
$p$. We choose $k_1$ such that $k_1v_t=k$ and that $v_t k_1/v_t k(t)$ is
infinite. Take $\Gamma$ to be the convex subgroup of $vK$ such that
$v_\Gamma=w$; in fact, $\Gamma$ is the minimal convex subgroup
containing $vt$.

Suppose $K$ admits a weak complement $R$. Then by Lemma~\ref{embdsrf}
the isomorphic image $K_\Gamma w$ of the subfield $K_\Gamma$ of $K$ is
dense in the valued residue field $(k_1,v_t)$. It follows from
Lemma~\ref{denseimm} that $v_t (K_\Gamma w)=v_t k_1\,$. Note that the
isomorphism $K_\Gamma\rightarrow K_\Gamma w$ preserves the prime field
$k$ of $K_\Gamma$. Since $v_t (K_\Gamma w)\ne\{0\}$, it follows from
Corollary~\ref{aaa} that $K_\Gamma w$ cannot be algebraic over the
trivially valued subfield $k$. Hence, $\trdeg K_\Gamma w|k=1$, and we
take some $t'\in K_\Gamma$ such that $t'w$ is transcendental over $k$.
It follows that $K_\Gamma| k(t')$ is algebraic. As $v_t k_1/v_t k(t)$ is
infinite, Lemma~\ref{degvr} shows that $K_\Gamma w |k(t)$ and hence also
$K_\Gamma|k(t')$ must be an infinite extension.

Since $\trdeg K|k=3$, we have that $\trdeg K|k(t')=2$. Let $\{x',y'\}$
be a transcendence basis for this extension. Because the algebraic
extension $K_\Gamma| k(t')$ is linearly disjoint from the purely
transcendental extension $k(t',x',y')|k(t')$, the extension
$K_\Gamma(x',y')|k(t',x',y')$ is infinite. But it is contained in the
finite extension $K|k(t',x',y')$. This contradiction shows that $K$
cannot admit weak complements. Note that by construction,
\[Kv\>=\>k_1v_t\>=\>k\>\subset\> K\;.\]
\end{example}

\begin{example} \rm                         \label{e2}
In the foregoing example, take $k=\Q$. By Lemma~\ref{liftord}, there is
an ordering $<$ on the rational function field $K=k(t,x,y)$ which is
compatible with the valuation $v$. Then $(K,<)$ does not admit an
integer part, because any such integer part would be a weak complement
for $v$.
\end{example}

\begin{example} \rm                         \label{e3}
In Example~\ref{e1}, take $k=\Q$. By Lemma~\ref{liftord} there is an
ordering on $k(t)$ compatible with the $v_t$-adic valuation. The real
closure $k(t)^{\rm rc}$ of $k(t)$ with respect to this ordering is a
countably generated infinite algebraic extension of $k(t)$. So we may
take $k_1= k(t)^{\rm rc}$. The valuation $v_t$ extends to a valuation of
$k_1$ which is compatible with its ordering. Again by
Lemma~\ref{liftord} we may choose a lifting of the ordering of $k_1$ to
$K$ through the valuation $w$. This ordering on $K$ induces through
$v=w\circ v_t$ the same ordering on the residue field $k$ as the
ordering on $k(t)^{\rm rc}$ induces through $v_t\,$; in particular, we
find that the chosen ordering on $K$ is compatible with the valuation
$v$.

Now pick any positive integer $n$ and consider the $n$-real closure
$K^{\rm rc(n)}$ of $(K,<)$ as defined in [Bg]. It is encluded in the
real closure $K^{\rm rc}$ of $(K,<)$, so we can extend $w$ to the real
closure (cf.\ Lemma~\ref{extcv}) and then restrict it to
$K^{\rm rc(n)}$; the valuation so obtained is still compatible with the
ordering. As $K^{\rm rc}w =(Kw)^{\rm rc}=k(t)^{\rm rc}=Kw$ by
Lemma~\ref{extcv}, we have that $K^{\rm rc(n)}w= k(t)^{\rm rc}$. So
$w\circ v_t$ is an extension of $v$ to $K^{\rm rc(n)}$, and we denote it
again by $v$. As before, we see that it is compatible with the ordering.

Suppose $K^{\rm rc(n)}$ admits a weak complement $R$. We proceed as in
Example~\ref{e1}, with $K$ replaced by $L:=K^{\rm rc(n)}$. As
$L_\Gamma w$ is dense in $(Lw,v_t)= (k(t)^{\rm rc},v_t)$ we can infer
from Lemma~\ref{denseimm} that $v_t (L_\Gamma w)=v_t k(t)^{\rm rc}$.
This in turn is the divisible hull of $v_t k(t)=\Z$ (cf.\
Lemma~\ref{extcv}). Hence, $v_t(L_\Gamma w)=\Q$ and thus also
$vL_\Gamma=\Q$. But as in Example~\ref{e1} one shows that the relative
algebraic closure of $k(t')$ in $K(t')$ must be a finite extension $E$
of $k(t')$. Now $L_\Gamma$ lies in the relative algebraic closure $E'$
of $E$ in $L$, which is just the $n$-real closure of $E$. But the value
group of the $n$-real closure of $E$ is the $n$-divisible hull of $vE$,
which in turn is a finite extension of $vk(t')=\Z$. So the value group
of $E'$ is still isomorphic to the $n$-divisible hull of $\Z$. This
contradicts the fact that its subfield $L_\Gamma$ has value group $\Q$.
This contradiction proves that $K^{\rm rc(n)}$ does not admit weak
complements for its compatible valuation $v$, and therefore does not
admit integer parts.
\end{example}

In order to obtain an example where the valued field $(K,v)$ admits an
embedded residue field {\it and} a cross-section, we modify
Example~\ref{e1} as follows.

\begin{example} \rm                         \label{e4}
In our basic construction, we take $k=k_0(z)$ where $k_0$ is any
prime field and $z$ is transcendental over $k_0\,$. The
henselization $k_0(z)^h$ of $k_0(z)$ with respect to the $z$-adic
valuation $v_z$ is a countably generated separable-algebraic extension
of $k_0(z)$. Therefore, we may choose $k_1$ to be a countably generated
separable-algebraic extension of $k(t)$ such that $v_tk_1=\Z$ and
$k_1v_t=k_0(z)^h$ (cf.\ Theorem~2.14 of [K1]). Then we take
\[v'\>=\>v\circ v_z \>=\> w\circ v_t\circ v_z\>.\]
Let $\Gamma$ be the convex subgroup of $v'K$ such that $v'_\Gamma=w$;
now $\Gamma$ is the minimal convex subgroup containing $v't$. Suppose
that $(K,v')$ admits a weak complement $R$. Then by Lemma~\ref{embdsrf}
the isomorphic image $K_\Gamma w$ of the subfield $K_\Gamma$ of $K$ is
dense in the valued residue field $(k_1,v_t\circ v_z)$. The isomorphism
$K_\Gamma\rightarrow K_\Gamma w$ preserves the prime field $k_0$ of
$K_\Gamma$. From Lemma~\ref{denseimm} we infer that $v_t\circ v_z
(K_\Gamma w)=v_t\circ v_z (k_1)$. This value group has two non-trivial
convex subgroups, namely, itself and the smallest convex subgroup which
contains $v_t\circ v_z (z)$. We choose elements $t',z' \in K_\Gamma$
such that $v_t\circ v_z(t'w)> 0$ lies in the former, but not in the
latter, and $v_t\circ v_z(z'w)> 0$ lies in the latter. Then these two
values are rationally independent. Thus by Theorem~1 of [Br], Chapter
VI, \S10.3, $t'w,z'w$ are algebraically independent over the trivially
valued field $k_0\,$. But as $K_\Gamma w\subseteq k_1\,$, we must have
$\trdeg K_\Gamma w|k_0=2$. Hence $K_\Gamma w|k_0(t'w,z'w)$ is algebraic,
and so is $K_\Gamma|k_0(t',z')$.

By Lemma~\ref{densevw}, $K_\Gamma w$ is also dense in $(k_1,v_t)$. Hence
$(K_\Gamma w)v_t=k_1 v_t=k_0(z)^h$ by Lemma~\ref{denseimm}. Hence,
$z\in (K_\Gamma w)v_t$ and we can in fact choose $z'$ such that
$(z'w)v_t=z$. Consequently, $k_0(t'w,z'w)v_t=k_0(z)$ (cf.\ the already
cited Theorem~1 of [Br]). Since $(K_\Gamma w)v_t=k_0(z)^h$ is an
infinite extension of $k_0(z)$ by part a) of
Proposition~\ref{ffnothens}, it follows from Lemma~\ref{degvr} that
$K_\Gamma w$ is an infinite extension of $k_0(t'w,z'w)$. Thus,
$K_\Gamma$ is an infinite extension of $k_0(z',t')$.

Since $\trdeg K|k_0=4$ and $\trdeg k_0(z',t')|k_0=\trdeg k_0(z'w,t'w)|
k_0 =2$, we have that $\trdeg K|k_0(z',t')=2$. Let $\{x',y'\}$ be a
transcendence basis for this extension. Because the algebraic extension
$K_\Gamma| k_0(z',t')$ is linearly disjoint from the purely
transcendental extension $k_0(z',t',x',y')|k_0(z',t')$, the
extension $K_\Gamma(x',y')|k_0(z',t',x',y')$ is infinite. But it is
contained in the finite extension $K|k_0(z',t',x',y')$. This
contradiction shows that $(K,v')$ cannot admit weak complements.

The value group $v'K$ is the lexicographic product $wK\times
v_tk_1\times v_zk_0(z)\isom\Z\times \Z\times\Z$ since $v_z k_0(z)=v_z
k_0(z)^h=\Z$. This shows that $(K,v')$ admits a cross-section. The
residue field $Kv'=k_0$ is embedded in $K$. Note that $K$ is a rational
function field of transcendence degree 4 over its residue field.
\end{example}

\begin{example} \rm
In the foregoing example, take $k_0=\Q$. By Lemma~\ref{liftord}, there
is an ordering $<$ on the rational function field $K=k(t,x,y,z)$ which
is compatible with the valuation $v'$. Then $(K,<)$ does not admit an
integer part. Nevertheless, the valuation $v'$, which is the natural
valuation of the ordering $<$ since $Kv'=\Q$ is archimedean ordered,
admits an embedding of its residue field and a cross-section.
\end{example}

Finally, let us note that Proposition~\ref{ffnothens} shows:
\begin{proposition}
None of the valued fields in the above examples are henselian. Also, the
natural valuation of the example constructed by Boughattas in [Bg] is
not henselian.
\end{proposition}
\begin{proof}
It follows from part a) of Proposition~\ref{ffnothens} that the rational
function fields of Examples~\ref{e1} and~\ref{e2} are not henselian.
The fields of Example~\ref{e3} and Boughattas' example are $n$-real
closures of algebraic function fields. The ``$n$-algebraic closures''
of part b) of Proposition~\ref{ffnothens} are algebraic extensions
of the $n$-real closures. Since they are not henselian,
Lemma~\ref{algehens} shows that the same holds for the
$n$-real closures.
\end{proof}

\bn
{\bf References}
\newenvironment{reference}%
{\begin{list}{}{\setlength{\labelwidth}{5em}\setlength{\labelsep}{0em}%
\setlength{\leftmargin}{5em}\setlength{\itemsep}{-1pt}%
\setlength{\baselineskip}{3pt}}}%
{\end{list}}
\newcommand{\lit}[1]{\item[{#1}\hfill]}
\begin{reference}
\lit{[B--K--K]} {Biljakovic, D.\ -- Kotchetov, M.\ -- Kuhlmann, S.$\,$:
{\it Primes and irreducibles in truncation integer parts of real closed
fields}, this volume}
\lit{[Bg]} {Boughattas, S.$\,$: {\it R\'esultats optimaux sur
l'existence d'une partie enti\`ere dans les corps ordonn\'es}, J.\
Symb.\ Logic {\bf 58} (1993), 326--333}
\lit{[Br]} {Bourbaki, N.$\,$: {\it Commutative algebra}, Paris (1972)}
\lit{[E]} {Endler, O.$\,$: {\it Valuation theory}, Berlin (1972)}
\lit{[Fo]} {Fornasiero, A.$\,$: {\it Embedding Henselian fields in
power series}, preprint}
\lit{[Fu]} {Fuchs, L.$\,$: {\it Partially ordered algebraic
systems}, Pergamon Press, Oxford (1963)}
\lit{[K1]} {Kuhlmann, F.-V.: {\it Value groups, residue fields and bad
places of rational function fields}, Trans.\ Amer.\ Math.\ Soc.\
{\bf 356} (2004), 4559--4600}
\lit{[K2]} {Kuhlmann, F.--V.$\,$: Book in preparation. Preliminary
versions of several chapters available at:\n
{\tt http://math.usask.ca/$\,\tilde{ }\,$fvk/Fvkbook.htm}}
\lit{[KS]} {Kuhlmann, S.: {\it Ordered Exponential Fields}, The Fields
Institute Monograph Series, vol.\ 12, AMS Publications (2000)}
\lit{[L]} {Lang, S.$\,$: {\it The theory of real places},
Ann.\ of Math.\ {\bf 57} (1953), 378--391}
\lit{[M--R]} {Mourgues, M.\ - H.\ -- Ressayre, J.- P.$\,$:
{\it Integer parts Every real closed field has an Integer Part},
Journal of Symbolic Logic, {\bf 58} (1993), 641--647}
\lit{[M--S]} {MacLane, S.\ -- Schilling, O.F.G.$\,$: {\it
Zero-dimensional branches of rank 1 on algebraic varieties}, Annals of
Math.\ {\bf 40} (1939), 507--520}
\lit{[N]} {Neukirch, J.$\,$: {\it Algebraic number theory},
Springer, Berlin (1999)}
\lit{[P]} {Prestel, A.$\,$: {\it Lectures on Formally Real Fields},
Springer Springer Lecture Notes in Math.\ {\bf 1093},
Berlin--Heidelberg--New York--Tokyo (1984)}
\lit{[R]} {Ribenboim, P.$\,$: {\it Th\'eorie des valuations}, Les
Presses de l'Uni\-versit\'e de Mont\-r\'eal (1964)}
\lit{[W]} {Warner, S.$\,$: {\it Topological fields}, Mathematics
studies {\bf 157}, North Holland, Amsterdam (1989)}
\lit{[Z--S]} {Zariski, O.\ -- Samuel, P.$\,$: {\it Commutative
Algebra}, Vol.\ II, New York--Heidel\-berg--Berlin (1960)}
\end{reference}
\adresse

\end{document}